\documentclass[preprint,12pt]{elsarticle}

\usepackage{epsfig}

\usepackage{amssymb}
\usepackage{amsmath}
\usepackage{amsthm}

\usepackage{geometry}
\usepackage{ulem}
\usepackage[T1]{fontenc}

\usepackage{setspace}
\usepackage{tikz}
\usepackage{mathdots}
\usepackage{algorithm}
\usepackage[tworuled,algo2e]{algorithm2e}

\DeclareMathAlphabet\mathbfcal{OMS}{cmsy}{b}{n}

\foreach \x in {a,...,z}{
    \expandafter\xdef\csname bf\x \endcsname{\noexpand\ensuremath{\noexpand\mathbf{\x}}} 
}
\foreach \x in {A,...,Z}{
    \expandafter\xdef\csname bf\x \endcsname{\noexpand\ensuremath{\noexpand\mathbf{\x}}} 
    \expandafter\xdef\csname bb\x \endcsname{\noexpand\ensuremath{\noexpand\mathbb{\x}}} 
    \expandafter\xdef\csname cal\x \endcsname{\noexpand\ensuremath{\noexpand\mathcal{\x}}} 
}

\newcommand{\T}{{\sf T}} 
\def\rank{{\mathrm{rank}}} 
\def\rankL{{\mathrm{rank_L}}} 
\def\rankR{{\mathrm{rank_R}}} 
\def\spanL{{\mathrm{span_L}}} 
\def\spanR{{\mathrm{span_R}}} 

\newcommand{\dmul}{\cdot_{\scriptscriptstyle\triangleright}}
\newcommand{\imul}{\cdot_{\scriptscriptstyle\triangleleft}}
\newcommand{\chid}{\chi_{\scriptscriptstyle\triangleright}}
\newcommand{\chii}{\chi_{\scriptscriptstyle\triangleleft}}

\newcommand{\ci}{\bfi} 
\newcommand{\cj}{\bfj} 
\newcommand{\ck}{\bfk} 
\newcommand{\Ci}{\bbC_{\bfi}} 

\newtheorem{definition}{Definition}
\newtheorem{theorem}[definition]{Theorem}
\newtheorem{proposition}[definition]{Proposition}
\newtheorem{lemma}[definition]{Lemma}
\newtheorem{corollary}[definition]{Corollary}
\newtheorem{remark}[definition]{Remark}
\newtheorem{example}[definition]{Example}

\journal{}

\begin{document}

\begin{frontmatter}

\title{On the rank of quaternion Hankel matrices} 

\author[labelGIPSA]{Philippe Flores\corref{cor1}}
\ead{flores.philipe@gmail.com}
\cortext[cor1]{Corresponding author.}
\author[labelCRAN]{Julien Flamant}
\author[labelGIPSA]{Nicolas Le Bihan}
\affiliation[labelGIPSA]{organization={CNRS, GIPSA-lab, Universite Grenoble Alpes, Grenoble INP},
addressline={11 Rue des Mathématiques},
            city={Saint-Martin-d'Hères},
            postcode={38402},
            country={FRANCE}}
\affiliation[labelCRAN]{organization={Université de Lorraine, CNRS},
            addressline={CRAN},
            city={F-54000},
            postcode={Nancy},
            country={France}}

\begin{abstract}
    This paper discusses the left and right ranks of quaternion matrices with Hankel structure.
    While they are in general different for arbitrary quaternion matrices, we show that the left and right ranks of quaternion Hankel matrices are equal.
    Moreover, we establish the relation between Hankel matrices and the existence of linear recurrence relations with quaternion coefficients and discuss some practical implications for computational methods relying on low-rank properties of quaternion Hankel matrices. 
\end{abstract}

\end{frontmatter}

\section{Introduction}

Structured matrices such as Hankel matrices arise in many fields including signal processing, control theory, system identification or time series analysis, among others \cite{arbib_foundations_1980}.
Formally, a Hankel matrix $\bfH$ with $F$ rows and $G$ columns has constant anti-diagonals, i.e., 
\begin{equation}
\label{eq:hankel_def}
\bfH = \begin{bmatrix}
        h_1 & \cdots & h_F & \cdots & h_G \\
        \vdots & \iddots & & \iddots & \vdots \\
        h_F & \cdots & h_G & \cdots & h_N
    \end{bmatrix}, \quad N = F+G-1,
\end{equation}
where elements $\{h_n\}^N_{n=1}$ belong to an arbitrary ring $\bbK$.
By construction, \eqref{eq:hankel_def} defines a one-to-one map $\calH_F: \bbK^{N} \to \bbK^{F\times G}$ such that $\bfH = \calH_F(\bfh)$ where $\bfh = [h_1, \ldots, h_N]^\top$.
For Hankel matrices over a field, the following result is classical.
\begin{theorem}[see e.g., {\cite[Corollary~5.1]{heinig_algebraic_2013}} or {\cite[Lemma~3.1]{gillard_hankel_2022}}] \label{thm:hankel_field} 
    Let $\bbK = \bbR$ or $\bbC$ and let $\bfH\in\bbK^{F\times G}$ be a Hankel matrix. 
    Then, there exists an integer $S\leq \left\lfloor\frac{F+G}2\right\rfloor$ such that $\rank(\bfH) = \min(F,G,S)$.
\end{theorem}
This result exists in different forms, see for instance \cite[Theorem 4]{al_hankel_2017}. 
In particular, it is central for computational methods relying on low-rank properties of Hankel matrices, such as Hankel structured low-rank approximation \cite{gillard_hankel_2022} or singular spectrum analysis \cite{golyandina_analysis_2001}.

Up to our knowledge, there is no equivalent to Theorem \ref{thm:hankel_field} for quaternion Hankel matrices.
Indeed, while the study of matrices with quaternions entries has attracted a growing interest for the past decades \cite{zhang_quaternions_1997,rodman2014topics}, the use and study of Hankel matrices with quaternion entries has received little interest so far.
This paper aims at filling this gap by extending such result to the quaternion setting.
The main challenge lies in handling the non-commutativity of quaternion multiplication, which for example requires the distinction between left and right linearity properties (such as the matrix rank).

This paper investigates the left and right ranks of quaternion Hankel matrices. 
In particular, we first generalize Theorem \ref{thm:hankel_field} to the quaternion setting in Theorem \ref{thm:main_theo}.
This result permits to show that for any quaternion Hankel matrix, its left and right ranks are equal (Theorem \ref{thm:equalityofranks}).
Our approach relies on the careful adaptation of the approach exposed in \cite{al_hankel_2017} for Hankel matrices with real or complex coefficients using the concept of infinite Hankel matrices.
However, the non-commutativity of quaternion algebra and the distinction between left and right linear combinations necessitate entirely new proofs for several key results.
We further highlight the relation between the rank of quaternion Hankel matrices and the existence of left (or right) Linear Recurrence Relations (LRRs) over the generating quaternion sequence. 

This paper is organized as follows.
Section \ref{sec:prelim} introduces relevant quaternion linear algebra notions.
Section \ref{sec:quaternion_hankel} then presents the main results on the rank of quaternion Hankel matrices and their implications on quaternion LRRs. 
Discussions and concluding remarks are held in Section \ref{sec:discussion}.

\section{Preliminaries} \label{sec:prelim}

The set of quaternions $\bbH$ is a four-dimensional normed division algebra over the real numbers $\bbR$.
It is usually defined as $\bbH = \left\{ a+b\ci+c\cj+d\ck \;\middle|\; (a,b,c,d) \in\bbR^4 \right\}$  where the canonical imaginary units $\{\ci, \cj, \ck\}$ satisfy the well-known relations \cite{hamilton_theory_1844}:
\begin{equation*}
    \ci^2 = \cj^2 = \ck^2 = \ci\cj\ck = -1.
\end{equation*}
These relations imply that, for instance, $\ci\cj = -\cj\ci$, and therefore the product of two quaternions is  non-commutative, i.e., for $p, q \in \bbH$, $p q \neq q p$ in general.
Quaternion-valued vectors and matrices are readily defined \cite{zhang_quaternions_1997,rodman2014topics}, as usual vectors and matrices with quaternions coefficients, and denoted as $\bfa = \left[a_n\right]^N_{n=1} \in\bbH^N$ and $\bfA = \left[a_{mn}\right]^{M,N}_{m,n=1} \in\bbH^{M\times N}$, respectively.

The non-commutativity of quaternion multiplication requires to adapt the definition of several key linear algebra concepts such as matrix rank.
Indeed, for a quaternion-valued matrix $\bfA\in\bbH^{M\times N}$, two notions of ranks exist, namely left and right ranks \cite{flamant2024multilinear}:
\begin{equation*}
    \rankL(\bfA) := \dim(\spanL(\bfA)), \quad \text{and} \quad \rankR(\bfA) := \dim(\spanR(\bfA)),
\end{equation*}
where $\spanL$ and $\spanR$ denote the left and right linear spans of columns of $\bfA = \left[\bfa_1\: \ldots\:\bfa_N\right]$:
\begin{equation*}
    \spanL(\bfA) := \left\{ \sum\limits_{n=1}^N \alpha_n \bfa_n\;\middle|\; \boldsymbol{\alpha}\in\bbH^N \right\}, \quad \text{and} \quad \spanR(\bfA) := \left\{ \sum\limits_{n=1}^N \bfa_n\beta_n \;\middle|\; \boldsymbol{\beta}\in\bbH^N \right\}.
\end{equation*}
In general, the left and right ranks of a matrix are different, see e.g., \cite[Example 7.3]{zhang_quaternions_1997}.
\begin{remark}\label{rk:rowrank_vs_colrank}
   For convenience, we mostly focus in this paper on linear spans of columns to study the left and right rank properties of quaternion matrices.
   However, one can equivalently define the rank in terms of linear spans of rows. 
   In particular, if a quaternion matrix has right rank $S$, by definition the right span of its columns has dimension $S$, which in turn is also the dimension defined by the left linear span of its rows. 
   The converse holds as well. 
   See also \cite{zhang_quaternions_1997} for further discussion. 
\end{remark}
The necessary distinction between left and right ranks motivates the need for introducing two types of matrix products between quaternion matrices. 
\begin{definition}[Matrix products for quaternion matrices]
    Let $\bfA\in\bbH^{I\times J}$ and $\bfB\in\bbH^{J\times K}$. 
    The direct and reverse products of $\bfA$ with $\bfB$ are defined as:  
    \begin{equation*}
        \bfA\dmul\bfB := \left[ \sum\limits_{j=1}^J a_{ij}b_{jk} \right]^{I,K}_{i,k=1} \quad \text{and} \quad \bfA\imul\bfB := \left[ \sum\limits_{j=1}^J b_{jk}a_{ij} \right]^{I,K}_{i,k=1}.
    \end{equation*}
\end{definition}
In particular, for a given matrix $\bfA \in \bbH^{M\times N}$ and a vector $\bfb \in \bbH^{N}$, notation $\bfA \dmul \bfb$ refers to the right linear combination of the columns of $\bfA$ using entries of $\bfb$ as coefficients; conversely, $\bfA\imul\bfb$ indicates a left linear combination of such columns.
For arbitrary $\bfA, \bfB$ matrices, the products $\bfA\dmul\bfB$ and $\bfA\imul\bfB$ are of the same size.
However, in general $\bfA\dmul\bfB\neq\bfA\imul\bfB$ due to non-commutativity of quaternion multiplication.
Nonetheless, one has $\left(\bfA\dmul\bfB\right)^\T = \bfB^\T\imul\bfA^\T$, where $\bfA^\T$ denotes the usual transpose of $\bfA$.

To study the rank of quaternion-valued matrices as in Lemma \ref{lem:symmetric_rank} below, two complex adjoint matrices are introduced, each one associated to its corresponding matrix product.
\begin{definition}[Complex adjoint matrices] \label{def:adjoint}
    Let $\bfA\in\bbH^{N\times M}$ and write $\bfA = \bfA_1+\bfA_2\cj$ with $\bfA_1,\bfA_2\in\Ci^{N\times M}$ where $\Ci = \bbR\oplus\ci\bbR$.
    The direct and reverse complex adjoint matrices of $\bfA$ are the matrices $\chid(\bfA)$ and $\chii(\bfA)$ of size $2N\times 2M$ defined by:
    \begin{equation*}
        \chid(\bfA) := \begin{bmatrix}
            \bfA_1 & \bfA_2 \\ -\overline{\bfA_2} & \overline{\bfA_1}
        \end{bmatrix}, \quad \text{and} \quad 
        \chii(\bfA) := \begin{bmatrix}
            \bfA_1 & -\overline{\bfA_2} \\ \bfA_2 & \overline{\bfA_1}
        \end{bmatrix},
    \end{equation*}
    where $\overline{\bfA_1}$ denotes the complex conjugate of $\bfA_1$.
\end{definition}
In particular, it can be easily verified that $\chid(\bfA\dmul\bfB) = \chid(\bfA)\chid(\bfB)$ and similarly, $\chii(\bfA\imul\bfB) = \chii(\bfA)\chii(\bfB)$.
This important property permits to define an isomorphism between quaternion matrices equipped with one matrix product onto the set of complex matrices (with the corresponding adjoint block structure).

This property is used to prove a lemma on the ranks of symmetric quaternion matrices, which will turn to be useful for establishing the main results of this paper. 
\begin{lemma} \label{lem:symmetric_rank}
    Let $\bfA\in\bbH^{N\times N}$ be a symmetric matrix.
    Then, $\rankL(\bfA) = \rankR(\bfA)$.
\end{lemma}
\begin{proof}
    From \cite[Proposition 2.7]{flamant2024multilinear}, one has $\rank(\chid(\bfA)) = 2\rankR(\bfA)$ and $\rank(\chii(\bfA)) = 2\rankL(\bfA)$.
    For any quaternion matrix $\bfA$, by direct calculation it holds that $\chid(\bfA^\T) = \chii^\T(\bfA)$.
    Since $\bfA$ is symmetric, $\bfA^\top = \bfA$ and therefore $\chid(\bfA) = \chii^\top(\bfA)$.
    Since the rank of a complex matrix and its transpose are equal, we get $\rankL(\bfA) = \frac{1}{2} \rank(\chii(\bfA)) = \frac{1}{2}\rank(\chii^\top(\bfA)) = \frac{1}{2}\rank(\chid(\bfA)) = \rankR(\bfA)$.
\end{proof}

\section{The left and right ranks of quaternion Hankel matrices} \label{sec:quaternion_hankel} 

We first make an observation on the left and right ranks of square quaternion Hankel matrices. 
Indeed, such matrices are symmetric, and therefore by Lemma \ref{lem:symmetric_rank}, their left and right ranks are equal. 
\begin{proposition}\label{prop:squareHankel}
     Let $\bfH$ be a square quaternion Hankel matrix. Then $\rankL(\bfH) = \rankR(\bfH)$.
\end{proposition}

We now turn our attention to the precise characterization of the rank of quaternion Hankel matrices of arbitrary size using the concept of infinite Hankel matrices. 
The approach is inspired by and closely follows that of \cite{al_hankel_2017}, which addresses the case of Hankel matrices over a field.
However, extending this approach to the quaternion case requires careful handling of the non-commutativity of quaternion multiplication, which necessitates distinguishing between left and right linear combinations in both statements and proofs.
Furthermore, the lack of classical concepts (e.g., minors) in quaternion linear algebra makes it necessary to provide new proofs. 

Given $\{h_n\}_{n\geq1}$ an infinite quaternion-valued sequence, its associated infinite Hankel matrix $\bfH_\infty$ reads:
\begin{equation}
    \bfH_\infty = \begin{bmatrix}
        h_1 & \cdots & h_F & \cdots & h_G & \cdots \\
        \vdots & \iddots & & \iddots & \vdots \\
        h_F & \cdots & h_G & \cdots & h_N & \cdots \\
        \vdots & & \vdots & & \vdots & \ddots
    \end{bmatrix}.\label{eq:H_infty}
\end{equation}
It is readily observed that $\bfH_\infty$ is symmetric.
Therefore, by Lemma \ref{lem:symmetric_rank} one has $\rankL(\bfH_{\infty}) = \rankR(\bfH_\infty)$ and thus the left/right rank distinction becomes unnecessary, allowing us to drop the subscript.

The following lemma is easily adapted from \cite{al_hankel_2017} to the case of quaternion Hankel matrices.
\begin{lemma}[\cite{al_hankel_2017}, revisited]
    \label{lemma:alpin_revisited}
    Let $\bfH_{\infty}$ be an infinite Hankel matrix. The following statements are equivalent:
    \begin{enumerate}
        \item $\rank(\bfH_\infty) = S$ is finite;
        \item the first $S$ columns of $\bfH_\infty$ are left linearly independent and generate all columns of $\bfH_\infty$ as a left linear combination;
        \item the first $S$ columns of $\bfH_\infty$ are right linearly independent and generate all columns of $\bfH_\infty$ as a right linear combination.
    \end{enumerate}
\end{lemma}
\begin{proof} 
    It is sufficient to prove the equivalence between 1. and 2. as the equivalence between 1. and 3. follows directly by switching between left and right linear independence. We follow closely the proof given in \cite{al_hankel_2017}.
    Observe that the columns of $\bfH_\infty$ are obtained by a shifted sequence of the first column $\bfh_1$:
    \begin{equation*}
        \bfH_\infty = \begin{bmatrix}
            \bfh_1 & \varphi(\bfh_1) & \varphi^2(\bfh_1) & \cdots & 
        \end{bmatrix}
    \end{equation*}
    where $\varphi$ corresponds to the shifting operator such that $\varphi([h_1, h_2, \ldots]) = [h_2, h_3, \ldots]$.
    
    Now suppose that $\rank(\bfH_\infty) = S$ is finite. Since $\bfH_\infty$ is symmetric, this means that $\rankL(\bfH_\infty) = S$. 
    This implies that any set of $S+1$ columns is left linearly dependent.
    In particular, the vectors $\bfh_1, \varphi(\bfh_1),\ldots, \varphi^{S}(\bfh_1)$ are left linearly dependent.
    Because of the shifting nature of $\varphi$, $\varphi^{S}(\bfh_1) \in \spanL([\bfh_1, \varphi(\bfh_1),\ldots, \varphi^{S-1}(\bfh_1)])$. 
    By recurrence, for any $k\geq S$, $\varphi^{k}(\bfh_1) \in \spanL([\bfh_1, \varphi(\bfh_1),\ldots, \varphi^{S-1}(\bfh_1)])$ as well. 
    Since $\rankL(\bfH_\infty) = S$, the first $S$ columns are left linearly independent (otherwise we would have $\rankL(\bfH_\infty) < S$).

    Conversely, suppose that the first $S$ columns of $\bfH_\infty$ are left linearly independent and generate all columns of $\bfH_\infty$ as a left linear combination.
    While the left linear independence implies that $\rankL(\bfH_\infty)\geq S$, the fact that any column is generated by the first $S$ columns implies that $\rankL(\bfH_\infty)\leq S$ which finishes the proof.
\end{proof}

Let $\bfH^{[S]}$ be the upper-left corner submatrix of $\bfH_\infty$ of size $S\times S$.
The following theorem generalizes the Kronecker theorem to the case of quaternion Hankel matrices.
For matrices over a field, the proof of this result usually relies on minors \cite[Theorem 2]{al_hankel_2017}.
However, such notion that is not well-defined in the quaternion case due to the difficulties in defining determinants of matrices (see e.g., \cite{aslaksen2001quaternionic}).
Instead, the proof below  relies on a careful handling of left and right linearity dependence properties of columns. 
In particular, the equality between left and right ranks of infinite Hankel matrices plays a key role.  
\begin{theorem}\label{thm:kronecker_quat}
    Consider the infinite Hankel matrix $\bfH_\infty$ defined in \eqref{eq:H_infty}. If $\rank(\bfH_\infty) = S$ is finite, then $\rankL(\bfH^{[S]}) = \rankR(\bfH^{[S]}) = S$.
\end{theorem}
\begin{proof}
    Let $\bfH_\infty = [\bfh_k]_{k\geq1}$ and let $\bfh^{[S]}_k\in\bbH^{S}$ denote the restriction of $\bfh_k$ to its first $S$ entries. 
    Assume that $\rank(\bfH_\infty) = \rankL(\bfH_\infty) = \rankR(\bfH_\infty) = S$ is finite.
    First, observe that $\bfH^{[S]} = [\bfh^{[S]}_s]_{1\leq s\leq S}$ is symmetric so that, by Lemma \ref{lem:symmetric_rank}, $\rankL(\bfH^{[S]}) = \rankR(\bfH^{[S]})$. 
    Thus, it is sufficient to prove, for example, that $\rankL(\bfH^{[S]}) = S$.
    
    By Lemma \ref{lemma:alpin_revisited}, $\{\bfh_1, \ldots, \bfh_S\}$ are left linearly (respectively right linearly) independent and, for any $k> S$, $\bfh_k$ is a left (respectively right) linear combination of $\{\bfh_1, \ldots, \bfh_S\}$.
    Hence, there exists $\boldsymbol{\alpha}\in \bbH^S$ (respectively $\boldsymbol{\beta}\in \bbH^S$) such that $\bfh^{[S]}_k = \sum_{s=1}^S\alpha_s \bfh^{[S]}_s$ (respectively $\bfh^{[S]}_k = \sum_{s=1}^S \bfh^{[S]}_s\beta_s$).

    Now, suppose by contradiction that $\bfh_1^{[S]}, \ldots, \bfh_S^{[S]}$ are left linearly dependent.
    Thus, there exists $\boldsymbol{\gamma} \in \bbH^S \setminus \mathbf{0}$ such that $\sum_{s=1}^S \gamma_s\bfh_s^{[S]} = 0$.
    By the Hankel structure, this implies $\sum_{s=1}^S \gamma_sh_{n+s-1} = 0$ for $n=1, \ldots, S$. 
    Moreover, as mentioned above, the ($S+1$)-th column $\bfh_{S+1}$ of $\bfH_\infty$ can be obtained by right linear combination of the of $S$ first columns, and in particular $\bfh^{[S]}_{S+1} = \sum_{\ell=1}^S \bfh^{[S]}_\ell\beta_\ell$ or equivalently, $h_{n+S} = \sum_{\ell=1}^S h_{n+\ell-1}\beta_\ell$ for $n=1, \ldots, S$. 
    Therefore, $\sum^S_{s=1} \gamma_sh_{S+s} = \sum_{s=1}^S \gamma_s\sum_{\ell=1}^S h_{s+\ell-1}\beta_\ell = \sum_{\ell=1}^S(\sum_{s=1}^S \gamma_s h_{s+\ell-1})\beta_\ell = 0$.
    Since any column of $\bfH_\infty$ can be expressed as a right linear combination of the first $S$ columns, one obtains by recurrence that $\sum_{s=1}^S \gamma_sh_{n+s-1} = 0$ for any $n \geq 1$. 
    This equivalently means that $\sum_{s=1}^S \gamma_s\bfh_s = 0$ and therefore, the infinite columns $\bfh_1, \ldots, \bfh_S$ are left linearly dependent.
    Since by Lemma \ref{lemma:alpin_revisited},  $\lbrace\bfh_1, \ldots, \bfh_S\rbrace$ form a basis for the column space of $\bfH_\infty$, this implies that $\rank(\bfH_\infty) < S$, which contradicts the initial assumption that $\rank(\bfH_\infty) = S$.
    Therefore, $\{\bfh_1^{[S]}, \ldots, \bfh^{[S]}_S\}$ are left linearly independent and $\rankL(\bfH^{[S]}) = S$. 
\end{proof}

Following \cite{al_hankel_2017}, the study of the rank of finite Hankel matrices requires the introduction of the following triangle table:
\begin{equation}
    \tau(h_1, \ldots, h_N) = \begin{pmatrix}
        h_1 & \cdots & h_n & \cdots & h_N \\
        \vdots & \cdot & \vdots & \iddots \\
        h_n & \cdots & h_N \\
        \vdots & \iddots & \\
        h_N
    \end{pmatrix},
    \label{eq:table}
\end{equation}
which is contained in any infinite extension $\bfH_\infty$ of the sequence $[h_1,\ldots, h_N]$.
\begin{definition}
    The $n$-th column $(h_n\: \cdots\: h_N)$ of the table \eqref{eq:table} (of length $N-n+1$) is said to be a \textit{prefix left linear combination} (respectively right linear) of preceding columns if it is a left linear (respectively right linear) combination of the preceding columns truncated to their first $N-n+1$ entries.
\end{definition}
\begin{example}
    For the sequence $[0,1,1,2,3,5]$, any column (except the first one) of the table 
    \begin{equation*}
        \tau(0,1,1,2,3,5) = \begin{pmatrix}
            0 & 1 & 1 & 2 & 3 & 5 \\
            1 & 1 & 2 & 3 & 5 \\
            1 & 2 & 3 & 5 \\
            2 & 3 & 5 \\
            3 & 5 \\
            5 \\
        \end{pmatrix}
    \end{equation*}
    is a prefix linear combination of the preceding columns. For instance, it holds that $(2\: 3\: 5)= (1\: 2\: 3) + (1\: 1\: 2)$.
\end{example}

\begin{definition} \label{def:m}
    For a sequence of quaternions $[h_1,\ldots, h_N]$, let $m_\mathrm{L}$ (respectively $m_\mathrm{R}$) be the maximal number such that none of the columns of the table \eqref{eq:table} with index $i \leq m_\mathrm{L}$ (respectively $i\leq m_\mathrm{R}$) is a prefix left linear (respectively right linear) combination of the preceding columns.
\end{definition}

\begin{lemma} \label{lem:inf_extension}
    The smallest possible left rank (respectively right rank) of an infinite extension $\bfH_\infty$ is equal to $m_\mathrm{L}$ (respectively $m_\mathrm{R}$).
\end{lemma}
\begin{proof}
    We start by considering the left rank case.
    Firstly, the left rank of any infinite extension $\bfH_\infty$ (i.e.,  containing the table \eqref{eq:table} as its upper left corner) is at least $m_\mathrm{L}$.
    Otherwise, if $\rankL(\bfH_\infty)<m_\mathrm{L}$, there would be a column $\bfh_i$ ($i\leq m_\mathrm{L}$) such that it is a left linear combination of the preceding columns.
    Thus, by identification of its first entries, $(h_i, \ldots, h_N)$ would be a prefix left linear combination of preceding columns which contradicts the definition of $m_\mathrm{L}$.

    Secondly, let us prove the existence of an infinite quaternion sequence such that its Hankel matrix $\bfH_\infty$ is of left rank $m_\mathrm{L}$.
    We consider the two following cases.
    
    On one hand, suppose $m_\mathrm{L} = N$. 
    Note that this means that $h_n = 0$ for $n = 1,\ldots, N-1$ and $h_N \neq 0$.
    Then, by choosing any $h_{N+1}\in\bbH^*$, it is possible to write $h_{N+1} = \left(h_{N+1}h^{-1}_N\right)h_N$. This relation permits to define an order-1 left LRR such that, for all $n>N$:
    \begin{equation*}
        h_n = \left(h_{N+1}h^{-1}_N\right) h_{n-1}.
    \end{equation*}
    
    On the other hand, consider that $m_\mathrm{L} <N$.
    Therefore, by definition of $m_\mathrm{L}$, the column $(h_{\mathrm{L}+1}, \ldots, h_N)$ of \eqref{eq:table} is a left linear combination of its preceding columns.
    By defining $(\alpha_{m_\mathrm{L}},\ldots, \alpha_N)\in\bbH^{N-m_\mathrm{L}}$ the coefficients of such left linear combination, it is possible to construct an infinite sequence following this left LRR of order $(N-m_\mathrm{L})$:
    \begin{equation*}
        \text{For $n\geq 1$,} \quad h_{N+n} = \alpha_{m_\mathrm{L}} h_{N+n-1} +\cdots+ \alpha_N h_{n+m_\mathrm{L}}.
    \end{equation*}
    Finally, for both infinite constructions, it holds that the ($m_\mathrm{L}+1$)-th column of $\bfH_\infty$ is a left linear combination of preceding ones.
    By Lemma \ref{lemma:alpin_revisited}, the rank of $\bfH_\infty$ is then equal to $m_\mathrm{L}$.

    By symmetry of reasoning for the right rank case, it is possible to construct right LRRs that define infinite matrices $\bfH_\infty$ of right ranks equal to $m_{\rm R}$.
\end{proof}

\begin{corollary} \label{cor:ml_equals_mr}
    For any quaternion sequence $[h_1,\ldots, h_N]$, it holds that $m_{\rm L} = m_{\rm R}$.
\end{corollary}
\begin{proof}
    By Lemma \ref{lem:inf_extension}, there exists $\bfH_\infty$ such that $\rankL(\bfH_\infty) = m_{\rm L}$ and $\bfH'_\infty$ such that $\rankR(\bfH'_\infty) = m_{\rm R}$. 
    Suppose by contradiction that $m_{\rm L} \neq m_{\rm R}$, for example $m_{\rm L} < m_{\rm R}$ without loss of generality.
    Lemma \ref{lem:symmetric_rank} states that $\rankR(\bfH_\infty) = m_{\rm L}$, hence proving the existence of an infinite matrix of right rank strictly smaller than $m_{\rm R}$, which is a contradiction.
\end{proof}

The next theorem is one of the main results of the paper, which generalizes \cite[Theorem 4]{al_hankel_2017} to the case of quaternion-valued Hankel matrices by exhibiting the left and right ranks for such matrices.
\begin{theorem} \label{thm:main_theo}
    Let $\bfH\in\bbH^{F\times G}$ be a Hankel matrix. 
    Then, $\rankL(\bfH) = \min(F,G,m_{\rm L},F+G-m_{\rm L})$ and $\rankR(\bfH) = \min(F,G,m_{\rm R},F+G-m_{\rm R})$ where $m_\mathrm{L}$ and  $m_\mathrm{R}$ are given by Definition \ref{def:m}.
\end{theorem}
\begin{proof}
    Similarly to Lemma \ref{lem:inf_extension}, the proof is restricted to the left case by symmetry of arguments with the right case.
    Following Lemma \ref{lem:inf_extension}, let us consider the infinite extension $\bfH_\infty$ of rank $m_\mathrm{L}$.
    In the following, we denote $\bfH^{[m_{\rm L}]}\in\bbH^{m_{\rm L}\times m_{\rm L}}$ the upper left corner square matrix of $\bfH_\infty$.
    By Theorem \ref{thm:kronecker_quat}, one has $\rankL(\bfH^{[m_{\rm L}]}) = m_{\rm L}$. 
    Similarly to \cite{al_hankel_2017}, the proof studies the three following cases.

    Case 1. Assume $m_{\rm L}<\min(F,G)$. This also implies that $m_{\rm L} \leq F + G - m_{\rm L}$. In this case, $\bfH^{[m_{\rm L}]}$ is a corner submatrix of $\bfH$. 
    Since $\rankL(\bfH^{[m_{\rm L}]}) = m_{\rm L}$, it holds that $\rankL(\bfH) \geq m_{\rm L}$. Because $\rankL(\bfH_\infty) = m_{\rm L}$ and $\bfH$ is also a submatrix of $\bfH_\infty$, it also holds that $\rankL(\bfH)\leq m_{\rm L}$ which gives $\rankL(\bfH) = m_{\rm L} = \min(F,G,m_{\rm L},F+G-m_{\rm L})$.

    Case 2. Suppose $\min(F,G)<m_{\rm L}<\max(F,G)$. First, let us consider $G<m_{\rm L}<F$.
    Since $\rankL(\bfH^{[m_{\rm L}]}) = m_{\rm L}>G$,  the first $G$ columns of $\bfH^{[m_{\rm L}]}$ (which are the truncated columns of $\bfH$ of length $m_{\rm L}$) are left linearly independent. 
    Hence, the columns of $\bfH$ are left linearly independent as well. This ensures that $\rankL(\bfH) = G = \min(F,G,m_{\rm L},F+G-m_{\rm L})$.
    Secondly, for $F < m_{\rm L} < G$,  $\rankL(\bfH^{[m_{\rm L}]}) = m_{\rm L}>F$ implies that the first $F$ rows of $\bfH^{[m_{\rm L}]}$ are right linearly independent (see Remark \ref{rk:rowrank_vs_colrank}). Therefore, the rows of $\bfH$ are right linearly independent as well, and $\rankL(\bfH) = F = \min(F,G,m_{\rm L},F+G-m_{\rm L})$.

    Case 3. Assume $m_{\rm L}>\max(F,G)$. Let us denote by $\bfH^{[F]}\in\bbH^{F\times m_{\rm L}}$ the matrix formed by the first $F$ rows of $\bfH^{[m_{\rm L}]}$. 
    Since $\rankL(\bfH^{[m_{\rm L}]})= m_{\rm L} > F$ the $F$ first rows of $\bfH^{[m_{\rm L}]}$ are right linearly independent (see Remark \ref{rk:rowrank_vs_colrank}) which ensures that $\rankL(\bfH^{[F]}) = F$.
    Then, it must be noted that the first $G$ columns of $\bfH^{[F]}$ define the matrix $\bfH$. As for any column of $\bfH^{[F]}$ of index $i>G$, it is not a left linear combination of preceding columns by definition of $m_{\rm L}$. Hence, two facts hold for those $m_{\rm L}-G$ columns. First, those columns are left linearly independent. Second, the left span of those columns does not intersect with the left span of columns of $\bfH$. Therefore, it holds that $\rankL(\bfH^{[F]}) = F = \rankL(\bfH)+m_{\rm L}-G$ which gives $\rankL(\bfH) = F+G-m_{\rm L} = \min(F,G,m_{\rm L},F+G-m_{\rm L})$.
\end{proof}
One can observe that this theorem is similar to Theorem \ref{thm:hankel_field} (which holds for Hankel matrices over a field), with the necessary distinction {\it a priori} between left and right ranks due to non-commutativity of quaternion algebra. 
However, the last theorem below states that this distinction is, perhaps not surprinsingly, not necessary. 
\begin{theorem} \label{thm:equalityofranks}
    For any Hankel matrix $\bfH\in\bbH^{F\times G}$, it holds that $\rankL(\bfH) = \rankR(\bfH)$.
\end{theorem}
\begin{proof}
    The proof is a straightforward application of Theorem \ref{thm:main_theo} and Corollary \ref{cor:ml_equals_mr}.
\end{proof}
Nonetheless, the existence of these left and right ranks has further implications on the properties of quaternions LRRs which are discussed in the next section.

\section{Discussions} \label{sec:discussion}

\paragraph{Remarks on Theorem \ref{thm:main_theo}}
Let us consider the problem of identifying a quaternion-valued LRR of order-$S$ from one associated Hankel matrix $\bfH$.
First, an obvious necessary condition to identify such LRR is that the dimension of $\bfH$ are greater than $m = S$ ($F,G\geq S$).
Figure \ref{fig:figure_rank} highlights this condition along with the rank "plateau" of Hankel matrices characterized by Theorem \ref{thm:main_theo}.
This phenomenon is well-known for real and complex Hankel matrices, see e.g., \cite{heinig_algebraic_2013,gillard_hankel_2022}.
\begin{figure}
    \centerline{\resizebox{0.55\linewidth}{!}{

\tikzset{every picture/.style={line width=0.75pt}} 

\begin{tikzpicture}[x=0.75pt,y=0.75pt,yscale=-1,xscale=1]

\draw [line width=3]    (100,320) -- (625,320) ;
\draw [shift={(630,320)}, rotate = 180] [color={rgb, 255:red, 0; green, 0; blue, 0 }  ][line width=3]    (20.77,-6.25) .. controls (13.2,-2.65) and (6.28,-0.57) .. (0,0) .. controls (6.28,0.57) and (13.2,2.66) .. (20.77,6.25)   ;
\draw [line width=3]    (120,340) -- (120,125) ;
\draw [shift={(120,120)}, rotate = 90] [color={rgb, 255:red, 0; green, 0; blue, 0 }  ][line width=3]    (20.77,-6.25) .. controls (13.2,-2.65) and (6.28,-0.57) .. (0,0) .. controls (6.28,0.57) and (13.2,2.66) .. (20.77,6.25)   ;
\draw [color={rgb, 255:red, 80; green, 80; blue, 80 }  ,draw opacity=1 ][line width=1.5]    (120,320) -- (240,200) ;
\draw [color={rgb, 255:red, 80; green, 80; blue, 80 }  ,draw opacity=1 ][line width=1.5]    (420,200) -- (540,320) ;
\draw [color={rgb, 255:red, 80; green, 80; blue, 80 }  ,draw opacity=1 ][line width=1.5]    (240,200) -- (420,200) ;
\draw [color={rgb, 255:red, 80; green, 80; blue, 80 }  ,draw opacity=1 ][line width=1.5]  [dash pattern={on 1.69pt off 2.76pt}]  (120,200) -- (240,200) ;
\draw [color={rgb, 255:red, 80; green, 80; blue, 80 }  ,draw opacity=1 ][line width=1.5]  [dash pattern={on 1.69pt off 2.76pt}]  (420,200) -- (540,200) ;
\draw [color={rgb, 255:red, 80; green, 80; blue, 80 }  ,draw opacity=1 ][line width=1.5]  [dash pattern={on 1.69pt off 2.76pt}]  (540,320) -- (540,200) ;
\draw [color={rgb, 255:red, 80; green, 80; blue, 80 }  ,draw opacity=1 ][line width=1.5]  [dash pattern={on 1.69pt off 2.76pt}]  (420,320) -- (420,200) ;
\draw [color={rgb, 255:red, 80; green, 80; blue, 80 }  ,draw opacity=1 ][line width=1.5]  [dash pattern={on 1.69pt off 2.76pt}]  (240,320) -- (240,200) ;

\draw (140,116) node [anchor=north west][inner sep=0.75pt]  [font=\LARGE] [align=left] {$\rank(\bfH)$};
\draw (98,192) node [anchor=north west][inner sep=0.75pt]  [font=\large] [align=left] {$S$};
\draw (238,331) node [anchor=north west][inner sep=0.75pt]  [font=\large] [align=left] {$S$};
\draw (418,332) node [anchor=north west][inner sep=0.75pt]  [font=\large] [align=left] {$S$};
\draw (608,332) node [anchor=north west][inner sep=0.75pt]  [font=\LARGE] [align=left] {$F$};

\end{tikzpicture}

    }}
    \caption{\centering Rank of a Hankel matrix with respect to its parameter size $F$. The observed plateau shows the existence of an order-$S$ LRR on the coefficients of the matrix $\bfH$. Therefore, the rank of $\bfH$ is at most $S$.}
    \label{fig:figure_rank}
\end{figure}
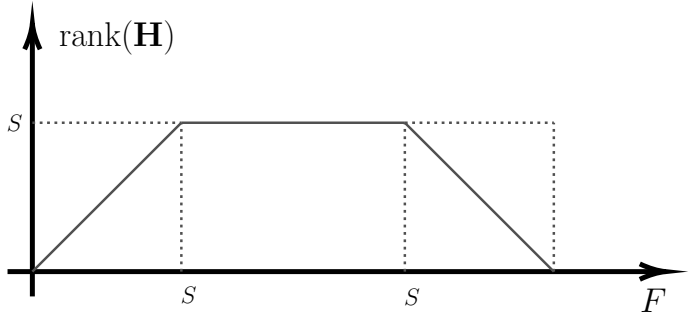 
However, even in the case of $F,G \geq S$, expressions (2) and (3) of Lemma \ref{lemma:alpin_revisited} show that there exists (at least) two generating quaternions LRRs, with left and right quaternion coefficients, respectively.

In other words, one may obtain suitable coefficients for a left LRR, even though the sequence followed in the first place a right LRR with different coefficients.
This may limit the usage of LRRs for practical applications whenever identifiability of the generating sequence is key. 

\paragraph{Switching from left to right LRRs and vice-versa}

Given a left LRR of order-$S$, there is no simple way, up to our knowledge, to obtain the corresponding right LRRs coefficients and vice-versa. 
However, in some cases this is possible by means of elementary quaternion algebra. 

Consider first the order-$1$ {\it left} LRR defined as $h_{n+1} = \mu h_n$ for all $n \in \mathbb{N}^*$ and $h_1 \in \bbH^*$. 
The elements of the sequence are thus $h_1, \mu h_1, \mu^2 h_1, \ldots$. 
One can easily check that the order-$1$ {\it right} LRR defined as $h'_{n+1} = h_n'\nu$ with $\nu = h_1^{-1} \mu h_1$ and $h_1' = h_1$ produces exactly the same sequence, i.e., $h_n = h_n'$ for any value of $n$. 
In other words, for order-$1$ quaternion LRRs, the recurrence coefficients $\mu, \nu$ are related through an involution by the initial point $h_1$. 

Such observation becomes more cumbersome as the order-$S$ increases.
For instance, consider two order-$1$ left LRRs such that $h_{n+1} = \mu_1 h_n$ and $q_{n+1} = \mu_2 q_n$, with $\mu_1, \mu_2 \in \bbH^*$ and $h_1, q_1 \in \bbH^*$. 
Define the sequence $w$ such that $w_n = h_n +q_n$ for any $n$.
Tedious calculations show that this sequence obeys to the order-$2$ left LRR 
$$\begin{cases}
\forall n \in \mathbb{N}^*, w_{n+2} = \eta w_{n+1} + \left[\mu_1 - \eta\right]\mu_1 w_n\\
w_1 = h_1 + q_1, w_2 = \mu_1 h_1 + \mu_2 q_1
\end{cases},\quad \eta := \left[\mu_2^2 - \mu_1^2\right]\left[\mu_2 - \mu_1\right]^{-1}. $$
By exploiting the left-right correspondence of order-$1$ quaternion LRR above, one can obtain the corresponding order-$2$ right LRR. 
Let $\nu_1 = h_1^{-1} \mu_1 h_1$ and $\nu_2 = q_1^{-1} \mu_2 q_1$. Similar calculations show that the above left recurrence can be written as a right recurrence such that
$$w_{n+2} = w_{n+1}\zeta + w_n \nu_1\left[\nu_1-\zeta\right],\quad \zeta := \left[\nu_2-\nu_1\right]^{-1}\left[\nu_2^2 - \nu_1^2\right],$$
for any $n \in \mathbb{N}^*$.
Switching between left and right recurrences is here made possible by the decomposition of the sequence $\lbrace w_n\rbrace_{n\geq 1}$ in sum of order-$1$ LRRs; however how to perform such correspondence in the general case -- that is, order-$S$ LRRs without the knowledge of a sum-of-order-$1$ LRRs -- remains an interesting open question. 

\paragraph{Other possible LRRs definitions} 
It is possible to define quaternion LRRs in a more general way, by considering both left and right coefficients.
For $S\geq1$ and $\{\mu_s,\nu_s\}^S_{s=1}\subset\bbH$, one can define a quaternion LRR as follows:
\begin{equation}
    \label{eq:lr_lrr}
    \forall n\in\bbN, \; h_{n+S} = \mu_Sh_{n+S-1}\nu_S+\cdots+\mu_1h_n\nu_1.
\end{equation}
Such a definition for quaternion LRRs with both left and right coefficients is more complicated to study, as the results proved previously on the rank of the associated Hankel matrix may not hold.
However, in some cases, such "double-sided" LRRs turn out to equivalent to a left or a right LRR, as shown by the following example.
\begin{example}
    Consider the following order-1 left and right LRR, as defined in Eq. \eqref{eq:lr_lrr}, with coefficients $\mu$ and $\nu$ being $p$-th roots of unity (i.e., such that $\mu^p=\nu^p = -1$):
    \begin{equation}
        \left\{ \begin{array}{l}
            \forall n\in\bbN,\;h_{n+1} = \mu h_n\nu \\
            h_1 \in\bbH^*
        \end{array}\right..
        \label{eq:lr_lrr_unitary}
    \end{equation}
    For any integer $n$, one has $h_{n+p} = h_n$. 
    Hence, the LRR \eqref{eq:lr_lrr_unitary} is an order-$p$ left or right LRR. 
\end{example}

\section{Acknowledgements}

This research was supported by the ANR project RICOCHET ANR-21-CE48-0013 and has received financial support from the CNRS through the MITI interdisciplinary programs.

The authors are grateful to C. de Seguins Pazzis for pointing out the use of an incorrect argument in an early version of the manuscript, which led to improvements in this work.

\end{document}